\documentclass[12pt]{amsproc}
\usepackage{hyperref}
\usepackage{geometry}
\usepackage{graphicx}
\usepackage{amsthm}
\usepackage{amssymb}

\theoremstyle{plain}
\newtheorem{thm}{Theorem}[section]
\newtheorem{cor}[thm]{Corollary}
\newtheorem{lem}[thm]{Lemma}
\newtheorem{prop}[thm]{Proposition}
\newtheorem{defn}[thm]{Definition}
\newtheorem{exa}[thm]{Example}

\begin{document}

\title{On Graded $2$-Absorbing Coprimary Submodules}

\author{Malik \textsc{Bataineh}}
\address{Department of Mathematics and Statistics, Jordan University of Science and Technology, Irbid, Jordan}
\email{msbataineh@just.edu.jo}

\author{Rashid \textsc{Abu-Dawwas}}
\address{Department of Mathematics, Yarmouk University, Irbid, Jordan}
\email{rrashid@yu.edu.jo}

\subjclass[2010]{Primary 16W50; Secondary 13A02}

\keywords{Graded strongly $2$-absorbing second submodules; graded $2$-absorbing primary submodules; graded strongly $2$-absorbing secondary submodules.}

\begin{abstract}
The aim of this article is to introduce the concept of graded $2$-absorbing coprimary submodules as a generalization of graded strongly $2$-absorbing second submodules, and explore some properties of this class. A non-zero graded $R$-submodule $N$ of a graded $R$-module $M$ is called a graded $2$-absorbing coprimary $R$-submodule if whenever $x, y$ are homogeneous elements of $R$ and $K$ is a graded $R$-submodule of $M$ such that $xyN\subseteq K$, then either $x$ or $y$ is in the graded radical of $(K :_{R} N)$ or $xy\in Ann_{R}(N)$. Several results have been achieved.
\end{abstract}

\maketitle

\section{Introduction}

Throughout this article, $G$ will be a group with identity $e$ and $R$ a commutative ring with nonzero unity $1$. Then $R$ is said to be $G$-graded if $R=\displaystyle\bigoplus_{g\in G} R_{g}$ with $R_{g}R_{h}\subseteq R_{gh}$ for all $g, h\in G$ where $R_{g}$ is an additive subgroup of $R$ for all $g\in G$. The elements of $R_{g}$ are called homogeneous of degree $g$. If $x\in R$, then $x$ can be written uniquely as $\displaystyle\sum_{g\in G}x_{g}$, where $x_{g}$ is the component of $x$ in $R_{g}$. It is known that $R_e$ is a subring of $R$ and $1\in R_e$. The set of all homogeneous elements of $R$ is $h(R)=\displaystyle\bigcup_{g\in G}R_{g}$. Let $I$ be an ideal of a graded ring $R$. Then $I$ is said to be graded ideal if $I=\displaystyle\bigoplus_{g\in G}(I\cap R_{g})$, i.e., for $x\in I$, $x=\displaystyle\sum_{g\in G}x_{g}$ where $x_{g}\in I$ for all $g\in G$. An ideal of a graded ring need not be graded. Assume that $M$ is a left unital $R$-module. Then $M$ is said to be $G$-graded if $M=\displaystyle\bigoplus_{g\in G}M_{g}$ with $R_{g}M_{h}\subseteq M_{gh}$ for all $g,h\in G$ where $M_{g}$ is an additive subgroup of $M$ for all $g\in G$. The elements of $M_{g}$ are called homogeneous of degree $g$. It is clear that $M_{g}$ is an $R_{e}$-submodule of $M$ for all $g\in G$. If $x\in M$, then $x$ can be written uniquely as $\displaystyle\sum_{g\in G}x_{g}$, where $x_{g}$ is the component of $x$ in $M_{g}$. The set of all homogeneous elements of $M$ is $h(M)=\displaystyle\bigcup_{g\in G}M_{g}$. Let $N$ be an $R$-submodule of a graded $R$-module $M$. Then $N$ is said to be graded $R$-submodule if $N=\displaystyle\bigoplus_{g\in G}(N\cap M_{g})$, i.e., for $x\in N$, $x=\displaystyle\sum_{g\in G}x_{g}$ where $x_{g}\in N$ for all $g\in G$. An $R$-submodule of a graded $R$-module need not be graded. For more details and terminology, see \cite{Hazart, Nastasescue}.

\begin{lem}\label{1}(\cite{Farzalipour}, Lemma 2.1) Let $R$ be a $G$-graded ring and $M$ be a $G$-graded $R$-module.

\begin{enumerate}

\item If $I$ and $J$ are graded ideals of $R$, then $I+J$ and $I\bigcap J$ are graded ideals of $R$.

\item If $N$ and $K$ are graded $R$-submodules of $M$, then $N+K$ and $N\bigcap K$ are graded $R$-submodules of $M$.

\item If $N$ is a graded $R$-submodule of $M$, $r\in h(R)$, $x\in h(M)$ and $I$ is a graded ideal of $R$, then $Rx$, $IN$ and $rN$ are graded $R$-submodules of $M$. Moreover, $(N:_{R}M)=\left\{r\in R:rM\subseteq N\right\}$ is a graded ideal of $R$.
\end{enumerate}
\end{lem}

In particular, $Ann_{R}(M)=(0_{M}:_{R}M)$ is a graded ideal of $R$. Let $M$ be a graded $R$-module, $N$ be a graded $R$-submodule of $M$ and $x\in h(R)$. Then $(N:_{M}x)=\left\{m\in M:xm\in N\right\}$ is a graded $R$-submodule of $M$, see (\cite{Refai Dawwas}, Lemma 2.9). Let $P$ be a proper graded ideal of $R$. Then the graded radical of $P$ is $Grad(P)$, and is defined to be the set of all $r\in R$ such that for each $g\in G$, there exists a positive integer $n_{g}$ for which $r_{g}^{n_{g}}\in P$. One can see that if $r\in h(R)$, then $r\in Grad(P)$ if and only if $r^{n}\in P$ for some positive integer $n$. In fact, $Grad(P)$ is a graded ideal of $R$, see \cite{Refai Hailat}.

In \cite{Zoubi Dawwas Ceken}, graded $2$-absorbing ideals as a generalization of graded prime ideals have been defined. A proper graded ideal $P$ of $R$ is said to be graded $2$-absorbing if whenever $a, b, c\in h(R)$ such that $abc\in P$, then either $ab\in P$ or $ac\in P$ or $bc\in P$. Al-Zoubi and Sharafat in \cite{Zoubi Sharafat}, introduced the concept of graded $2$-absorbing primary ideals. This concept generalizes both graded primary ideals and graded $2$-absorbing ideals. A proper graded ideal $P$ of $R$ is called a graded $2$-absorbing primary ideal of $R$ if whenever $a, b, c\in h(R)$ with $abc\in P$, then $ab\in P$ or $ac\in Grad(P)$ or $bc\in Grad(P)$. Graded second submodules of graded modules over commutative rings were introduced by Ansari-Toroghy and Farshadifar in \cite{Ansari Farshadifar} as follows: A non-zero graded $R$-submodule $N$ of a graded $R$-module $M$ is said to be a graded second $R$-submodule if for all $a\in h(R)$, either $aN = \{0\}$ or $aN = N$. In this case, $P=Ann_{R}(N)$ is a graded prime ideal of $R$, and $N$ is called a graded $P$-second $R$-submodule of $M$. In \cite{Refai Dawwas}, a generalization of graded second submodules has been given as follows: A non-zero graded $R$-submodule $N$
of a graded $R$-module $M$ is said to be a graded strongly $2$-absorbing second $R$-submodule of $M$ if whenever $x, y\in h(R)$ and $K$ is a graded $R$-submodule of $M$ such that $xyN\subseteq K$, then $xN\subseteq K$ or $yN\subseteq K$ or $xy\in Ann_{R}(N)$.

In this article, we act in accordance with \cite{Ceken} to introduce the concept of graded $2$-absorbing coprimary $R$-submodules of a graded $R$-module $M$ as a generalization of graded strongly $2$-absorbing second submodules, and explore some properties of this class. A non-zero graded $R$-submodule $N$ of a graded $R$-module $M$ is called a graded $2$-absorbing coprimary $R$-submodule if whenever $x, y\in h(R)$ and $K$ is a graded $R$-submodule of $M$ such that $xyN\subseteq K$, then $x\in Grad((K :_{R} N))$ or $y\in Grad((K :_{R} N))$ or $xy\in Ann_{R}(N)$.

\section{Graded $2$-Absorbing Coprimary Submodules}

In this section, we introduce and study the concept of graded $2$-absorbing coprimary submodules.

\begin{defn}Let $M$ be a graded $R$-module and $N$ be a non-zero graded $R$-submodule of $M$. Then $N$ is said to be a graded $2$-absorbing coprimary $R$-submodule of $M$ if whenever $x, y\in h(R)$ and $K$ is a graded $R$-submodule of $M$ such that $xyN\subseteq K$, then $x\in Grad((K :_{R} N))$ or $y\in Grad((K :_{R} N))$ or $xy\in Ann_{R}(N)$.
\end{defn}

Clearly, every graded strongly $2$-absorbing second submodule is graded $2$-absorbing coprimary. The next example shows that the converse is not true in general.

\begin{exa}\label{Example 2.2} Consider the $\mathbb{Z}$-module $M = \mathbb{Z}_{4}\bigoplus \mathbb{Z}_{9}\bigoplus \mathbb{Z}_{5}$. Then by (\cite{Ceken}, Example 2.2), $N =\mathbb{Z}_{4}\bigoplus \mathbb{Z}_{9}\bigoplus\{0\}$ is a $2$-absorbing coprimary submodule of $M$. So, if we consider the trivial graduation for $M$, $N$ will be  a graded $2$-absorbing coprimary submodule of $M$. On the other hand, $N$ is not graded strongly $2$-absorbing second submodule of $M$ by (\cite{Refai Dawwas}, Corollary 3.4).
\end{exa}

\begin{prop}\label{Proposition 2.6} Let $M$ be a graded $R$-module and $N$ be a graded $2$-absorbing coprimary $R$-submodule of $M$. If $K$ is a graded $R$-submodule of $M$ such that $N\nsubseteq K$, then $(K :_{R} N)$ is a graded $2$-absorbing primary ideal of $R$.
\end{prop}

\begin{proof}By Lemma \ref{1}, $(K:_{R}N)$ is a graded ideal of $R$. Let $x, y, z\in h(R)$ such that $xyz\in (K :_{R} N)$. Then $xyzN = z(xyN)\subseteq K$, and then $xyN\subseteq (K :_{M} z)$. Since $N$ is a graded $2$-absorbing coprimary $R$-submodule of $M$, $x\in Grad((K:_{M}z):_{R}N)$ or $y\in Grad((K:_{M}z):_{R}N)$ or $xy\in Ann_{R}(N)$. If $xy\in Ann_{R}(N)$, then $xyN=\{0\}\subseteq K$, and then $xy\in (K:_{R}N)$. If $x\in Grad((K:_{M}z):_{R}N)$, then $x^{n}zN\subseteq K$ for some positive integer $n$, and then $(xz)^{n}N\subseteq K$, which implies that $xz\in Grad((K:_{R}N))$. Similarly, if $y\in Grad((K:_{M}z):_{R}N)$, then $yz\in Grad((K:_{R}N))$. Hence, $(K :_{R} N)$ is a graded $2$-absorbing primary ideal of $R$.
\end{proof}
As a direct consequence of Proposition \ref{Proposition 2.6}, we have the following corollary.

\begin{cor}\label{Proposition 2.6 (1)} Let $M$ be a graded $R$-module and $N$ be a graded $2$-absorbing coprimary $R$-submodule of $M$. Then $Ann_{R}(N)$ is a graded $2$-absorbing primary ideal of $R$.
\end{cor}

\begin{cor}\label{Corollary 2.7} Let $M$ be a graded $R$-module and $N$ be a graded $2$-absorbing coprimary $R$-submodule of $M$. Then $Grad(Ann_{R}(N))$ is a graded $2$-absorbing ideal of $R$.
\end{cor}

\begin{proof}The result follows by Corollary \ref{Proposition 2.6 (1)} and (\cite{Zoubi Sharafat}, Theorem 2.3).
\end{proof}

\begin{prop}\label{Proposition 2.10} Let $M$ be a graded $R$-module and $N$ be a graded $2$-absorbing coprimary $R$-submodule of $M$. Then $aN$ is a graded $2$-absorbing coprimary $R$-submodule of $M$ for all $a\in h(R)-Ann_{R}(N)$.
\end{prop}

\begin{proof}Let $a\in h(R)-Ann_{R}(N)$. Then by Lemma \ref{1}, $aN$ is a graded $R$-submodule of $M$. Suppose that $x, y\in h(R)$ and $K$ is a graded $R$-submodule of $M$ such that $xy(aN)\subseteq K$. Then $xyN\subseteq(K :_{M}a)$. Since $N$ is a graded $2$-absorbing coprimary $R$-submodule of $M$, $x^{n}N\subseteq(K :_{M}a)$ for some positive integer $n$ or $y^{m}N\subseteq(K :_{M}a)$ for some positive integer $m$ or $xy\in Ann_{R}(N)$, and then $x^{n}(aN)\subseteq K$ or $y^{m}(aN)\subseteq K$ or $xy\in Ann_{R}(N)\subseteq Ann_{R}(aN)$, as needed.
\end{proof}

Let $M$ and $S$ be two $G$-graded $R$-modules. An $R$-homomorphism $f:M\rightarrow S$ is said to be a graded $R$-homomorphism if $f(M_{g})\subseteq S_{g}$ for all $g\in G$.

\begin{lem}\label{2}(\cite{Dawwas Bataineh}, Lemma 2.16) Suppose that $f:M\rightarrow S$ is a graded $R$-homomorphism of graded $R$-modules. If $K$ is a graded $R$-submodule of $S$, then $f^{-1}(K)$ is a graded $R$-submodule of $M$.
\end{lem}

\begin{lem}\label{3}(\cite{Atani Saraei}, Lemma 4.8) Suppose that $f:M\rightarrow S$ is a graded $R$-homomorphism of graded $R$-modules. If $L$ is a graded $R$-submodule of $M$, then $f(L)$ is a graded $R$-submodule of $f(M)$.
\end{lem}

\begin{prop}\label{Proposition 2.11}Suppose that $f:M\rightarrow S$ is a graded $R$-homomorphism of graded $R$-modules.
\begin{enumerate}
\item If $N$ is a graded $2$-absorbing coprimary $R$-submodule of $M$ such that $N\nsubseteq Ker( f )$, then $f (N)$ is a graded $2$-absorbing coprimary $R$-submodule of $S$.
\item If $K$ is a graded $2$-absorbing coprimary $R$-submodule of $S$ and $K\subseteq f (M)$, then $f^{-1}(K)$ is a graded $2$-absorbing coprimary $R$-submodule of $M$.
\end{enumerate}
\end{prop}

\begin{proof}
\begin{enumerate}
\item By Lemma \ref{3}, $f(N)$ is a graded $R$-submodule of $S$. Let $x, y\in h(R)$ and $K$ be a graded $R$-submodule of $S$ such that $xyf (N)\subseteq K$. Then $xyN\subseteq f^{-1}(K)$. Since $f^{-1}(K)$ is a graded $R$-submodule of $M$ by Lemma \ref{2} and $N$ is a graded $2$-absorbing coprimary $R$-submodule of $M$, we have $x^{n} N\subseteq f^{-1}(K)$ for some positive integer $n$ or $y^{m}N\subseteq f^{-1}(K)$ for some positive integer $m$ or $xy\in Ann_{R}(N)\subseteq Ann_{R}( f (N))$, and then $x^{n} f (N)\subseteq K$ or $y^{m} f (N)\subseteq K$ or $xy\in Ann_{R}( f (N))$, as required.
\item By Lemma \ref{2}, $f^{-1}(K)$ is a graded $R$-submodule of $M$. If $f^{-1}(K) = \{0\}$, then since $K\subseteq f (M)$, we have $f ( f^{-1}(K)) =K = \{0\}$, which is a contradiction. So, $f^{-1}(K)\neq \{0\}$. Let $x, y\in h(R)$ and $L$ be a graded $R$-submodule of $M$ such that $xyf^{-1}(K)\subseteq L$. Since $K\subseteq f (M)$, $xyK\subseteq f (L)$. Since $f(L)$ is a graded $R$-submodule of $S$ by Lemma \ref{3} and $K$ is a graded $2$-absorbing coprimary $R$-submodule of $S$, $x^{n}K\subseteq f (L)$ for some positive integer $n$ or $y^{m}K\subseteq f (L)$ for some positive integer $m$ or $xy\in Ann_{R}(K)$, and then $x^{n} f^{-1}(K) = f^{-1}(x^{n} K)\subseteq L$ or $y^{m} f^{-1}(K) = f^{-1}(y^{m}K)\subseteq L$ or $xy\in Ann_{R}(K) =Ann_{R}( f^{-1}(K))$, as needed.
\end{enumerate}
\end{proof}

\begin{lem}\label{Theorem 2.5}Let $M$ be a graded $R$-module and $N$ be a non-zero graded $R$-submodule of $M$. Then $N$ is a graded $2$-absorbing coprimary $R$-submodule of $M$ if and only if for any $x, y\in h(R)$, either $x^{n}N\subseteq xyN$ for some positive integer $n$ or $y^{m}N\subseteq xyN$ for some positive integer $m$ or $xy\in Ann_{R}(N)$.
\end{lem}

\begin{proof}Suppose that $N$ is a graded $2$-absorbing coprimary $R$-submodule of $M$. Let $x, y\in h(R)$. Then $xy\in h(R)$, and then by Lemma \ref{1}, $xyN$ is a graded $R$-submodule of $M$ such that $xyN\subseteq xyN$. Since $N$ is a graded $2$-absorbing coprimary $R$-submodule of $M$, $x^{n}N\subseteq xyN$ for some positive integer $n$ or $y^{m}N\subseteq xyN$ for some positive integer $m$ or $xy\in Ann_{R}(N)$. The converse is clear.
\end{proof}

Let $M$ be a graded $R$-module and $S$ be a multiplicative subset of $h(R)$. Then $S^{-1}M$ is a graded $S^{-1}R$-module with $$(S^{-1}M)_g=\left\{\frac{m}{s},m\in M_h, s\in S\cap R_{hg^{-1}}\right\},$$
$$(S^{-1}R)_g=\left\{\frac{a}{s},a\in R_h, s\in S\cap R_{hg^{-1}}\right\}.$$

\begin{prop}\label{Proposition 2.12} Let $M$ be a graded $R$-module, $S$ be a multiplicative subset of $h(R)$ and $N$ be a graded $R$-submodule of $M$ such that $S^{-1}N\neq\{0\}$. If $N$ is a graded $2$-absorbing coprimary $R$-submodule of $M$, then $S^{-1}N$ is a graded $2$-absorbing coprimary $S^{-1}R$-submodule of $S^{-1}M$.
\end{prop}

\begin{proof}Let $x, y\in h(R)$ and $s, t\in S$. Since $N$ is a graded $2$-absorbing coprimary $R$-submodule of $M$, by Lemma \ref{Theorem 2.5}, $x^{n}N\subseteq xyN$ for some positive integer $n$ or $y^{m}N\subseteq xyN$ for some positive integer $m$ or $xy\in Ann_{R}(N)$. If $xy\in Ann_{R}(N)$, then $\frac{x}{s}\frac{y}{t}(S^{-1}N) = \{0\}$. If $x^{n}N\subseteq xyN$, then $(\frac{x}{s})^{n}(S^{-1}N)=(\frac{x^{n}}{s^{n}})(S^{-1}N)=\frac{1}{s^{n}}S^{-1}(x^{n}N)\subseteq\frac{1}{s^{n}}S^{-1}(xyN)=S^{-1}(xyN)=
(\frac{1}{s}\frac{1}{t})(S^{-1}(xyN))=(\frac{x}{s}\frac{y}{t})(S^{-1}N)$. Similarly, if $y^{m}N\subseteq xyN$, then $(\frac{x}{s})^{m}(S^{-1}N)\subseteq (\frac{x}{s}\frac{y}{t})(S^{-1}N)$. Hence, by Lemma \ref{Theorem 2.5}, $S^{-1}N$ is a graded $2$-absorbing coprimary $S^{-1}R$-submodule of $M$.
\end{proof}

\begin{defn}Let $M$ be a $G$-graded $R$-module, $g\in G$ and $N$ be a non-zero graded $R$-submodule of $M$. Then $N$ is said to be a $g$-$2$-absorbing coprimary $R$-submodule of $M$ if whenever $x, y\in R_{g}$ and $K$ is a graded $R$-submodule of $M$ such that $xyN\subseteq K$, then $x\in Grad((K :_{R} N))$ or $y\in Grad((K :_{R} N))$ or $xy\in Ann_{R}(N)$.
\end{defn}

\begin{lem}\label{Lemma 2.3} Let $M$ be a $G$-graded $R$-module, $g\in G$, $I$ be a graded ideal of $R$ and $N$ be a $g$-$2$-absorbing coprimary $R$-submodule of $M$. If $x\in R_{g}$ and $K$ is a graded $R$-submodule of $M$ such that $IxN\subseteq K$, then $x\in Grad((K :_{R} N))$ or $I_{g}\subseteq Grad((K :_{R} N))$ or $I_{g}x\subseteq Ann_{R}(N)$.
\end{lem}

\begin{proof}Suppose that $x\notin Grad((K :_{R} N))$ and $I_{g}x\nsubseteq Ann_{R}(N)$. Then there exists $y\in I_{g}$ such that $xy\notin Ann_{R}(N)$, and then $xyN\subseteq K$. Since $N$ is $g$-$2$-absorbing coprimary, $y\in Grad((K:_{R}N))$. Let $z\in I_{g}$. Then $(y+z)xN\subseteq K$, and then $y + z\in Grad((K :_{R} N))$ or $(y + z)x\in Ann_{R}(N)$. If $y + z\in Grad((K :_{R} N))$, then $z\in Grad((K :_{R} N))$. If $(y + z)x\in Ann_{R}(N)$, then $xz\notin Ann_{R}(N)$. So, $zxN\subseteq K$ implies that $z\in Grad((K :_{R} N))$. Hence, $I_{g}\subseteq Grad((K :_{R} N))$.
\end{proof}

\begin{thm}\label{Lemma 2.4} Let $M$ be a $G$-graded $R$-module, $g\in G$, $I, J$ be two graded ideals of $R$ and $N$ be a $g$-$2$-absorbing coprimary $R$-submodule of $M$. If $K$ is a graded $R$-submodule of $M$ such that $IJN\subseteq K$, then $I_{g}\subseteq Grad((K :_{R} N))$ or $J_{g}\subseteq Grad((K :_{R} N))$ or $I_{g}J_{g}\subseteq Ann_{R}(N)$.
\end{thm}

\begin{proof}Suppose that $I_{g}\nsubseteq Grad((K :_{R} N))$ and $J_{g}\nsubseteq Grad((K :_{R} N))$. Then there exist $x\in I_{g}-Grad((K:_{R}N))$ and $y\in J_{g}-Grad((K:_{R}N))$. Let $z\in I_{g}$ and $w\in J_{g}$. Now, $xJN\subseteq K$, so by Lemma \ref{Lemma 2.3}, $xJ_{g}\subseteq Ann_{R}(N)$, and so $(I_{g}-Grad((K:_{R}N)))J_{g}\subseteq Ann_{R}(N)$. Similarly, $I_{g}y\subseteq Ann_{R}(N)$, and so $I_{g}(J_{g}-Grad((K:_{R}N)))\subseteq Ann_{R}(N)$. So, we conclude that $xy, xw, zy\in Ann_{R}(N)$. Now, $(x+z)(y+w)N\subseteq K$. Therefore, $x + z\in Grad((K :_{R} N))$ or $y + w\in Grad((K :_{R} N))$ or $(x + z)(y +w)\in Ann_{R}(N)$. If $x +z\in Grad((K :_{R} N))$, then $z\notin Grad((K :_{R} N))$. Thus $z\in I_{g}-Grad((K :_{R} N))$, which implies that $zw\in Ann_{R}(N)$. Similarly, if $y + w\in Grad((K :_{R} N))$, then $zw\in Ann_{R}(N)$. If $(x + z)(y + w) = xy + xw + yz + zw\in Ann_{R}(N)$, then $zw\in Ann_{R}(N)$. Hence, $I_{g} J_{g}\subseteq Ann_{R}(N)$.
\end{proof}

Graded comultiplication modules have been introduced by Toroghy and Farshadifar in \cite{Toroghy}; a graded $R$-module $M$ is said to be graded comultiplication if for every graded $R$-submodule $N$ of $M$, $N=(0:_{M}I)$ for some graded ideal $I$ of $R$, or equivalently, $N=(0:_{M}Ann_{R}(N))$.

\begin{prop}\label{Corollary 2.8} Let $M$ be a graded comultiplication $R$-module. If $N$ is a graded $2$-absorbing coprimary $R$-submodule of $M$ such that $Grad(Ann_{R}(N)) =Ann_{R}(N)$, then $N$ is a graded strongly $2$-absorbing second $R$-submodule of $M$. In particular, $N$ is a graded $2$-absorbing second $R$-submodule of $M$.
\end{prop}

\begin{proof}The result follows by Corollary \ref{Corollary 2.7} and (\cite{Refai Dawwas}, Proposition 3.7).
\end{proof}

Let $R_{1}$ and $R_{2}$ be two $G$-graded rings, $M_{1}$ be a $G$-graded $R_{1}$-module, $M_{2}$ be a $G$-graded $R_{2}$-module, $R=R_{1}\times R_{2}$ and $M=M_{1}\times M_{2}$. Then $M$ is a $G$-graded $R$-module by $M_{g}=(M_{1})_{g}\times (M_{2})_{g}$ for all $g\in G$.

\begin{thm}Let $R_{1}$ and $R_{2}$ be two $G$-graded rings, $M_{1}$ be a $G$-graded $R_{1}$-module, $M_{2}$ be a $G$-graded $R_{2}$-module, $R=R_{1}\times R_{2}$ and $M=M_{1}\times M_{2}$.
\begin{enumerate}
\item If $N=N_{1}\times N_{2}$ is a graded $2$-absorbing coprimary $R$-submodule of $M$, then $Ann_{R_{1}}(N_{1})$ is a graded primary ideal of $R_{1}$ and $Ann_{R_{2}}(N_{2})$ is a graded primary ideal of $R_{2}$.
\item If $N_{1}$ is a graded $R_{1}$-submodule of $M_{1}$, $N_{2}$ is a graded $R_{2}$-submodule of $M_{2}$ such that $Ann_{R_{1}}(N_{1})$ is a graded primary ideal of $R_{1}$, $Ann_{R_{2}}(N_{2})$ is a graded primary ideal of $R_{2}$ and $N=N_{1}\times N_{2}$, then $Ann_{R}(N)$ is a graded $2$-absorbing primary ideal of $R$.
\item If $N_{1}$ is a graded $2$-absorbing coprimary $R_{1}$-submodule of $M_{1}$ and $N=N_{1}\times\{0_{M_{2}}\}$, then $Ann_{R}(N)$ is a graded $2$-absorbing primary ideal of $R$.
\item If $N_{2}$ is a graded $2$-absorbing coprimary $R_{2}$-submodule of $M_{2}$ and $N=\{0_{M_{1}}\}\times N_{2}$, then $Ann_{R}(N)$ is a graded $2$-absorbing primary ideal of $R$.
\end{enumerate}
\end{thm}

\begin{proof}\begin{enumerate}
\item By (\cite{Saber}, Lemma 3.12), $N_{1}$ is a graded $R_{1}$-submodule of $M_{1}$ and $N_{2}$ is a graded $R_{2}$-submodule of $M_{2}$, and then $Ann_{R_{1}}(N_{1})$ is a graded ideal of $R_{1}$ and $Ann_{R_{2}}(N_{2})$ is a graded ideal of $R_{2}$. Since $N$ is graded $2$-absorbing coprimary, by Corollary \ref{Proposition 2.6 (1)}, $Ann_{R}(N)=Ann_{R_{1}}(N_{1})\times Ann_{R_{2}}(N_{2})$ is a graded $2$-absorbing primary ideal of $R$, and then by (\cite{Zoubi Sharafat}, Theorem 2.10), $Ann_{R_{1}}(N_{1})$ is a graded primary ideal of $R_{1}$ and $Ann_{R_{2}}(N_{2})$ is a graded primary ideal of $R_{2}$.
\item By (\cite{Saber}, Lemma 3.12), $N$ is a graded $R$-submodule of $M$, and then $Ann_{R}(N)$ is a graded ideal of $R$. Since $Ann_{R_{1}}(N_{1})$ is a graded primary ideal of $R_{1}$, $Ann_{R_{1}}(N_{1})\times R_{2}$ is a graded primary ideal of $R$. Similarly, $R_{1}\times Ann_{R_{2}}(N_{2})$ is a graded primary ideal of $R$. So, by (\cite{Zoubi Sharafat}, Lemma 2.9), $Ann_{R_{1}}(N_{1})\bigcap Ann_{R_{2}}(N_{2})=Ann_{R_{1}}(N_{1})\times Ann_{R_{2}}(N_{2})=Ann_{R}(N)$ is a graded $2$-absorbing primary ideal of $R$.
\item By (\cite{Saber}, Lemma 3.12), $N$ is a graded $R$-submodule of $M$, and then $Ann_{R}(N)$ is a graded ideal of $R$. Since $N_{1}$ is graded $2$-absorbing coprimary, by Corollary \ref{Proposition 2.6 (1)}, $Ann_{R_{1}}(N_{1})$ is a graded $2$-absorbing primary ideal of $R_{1}$, and then $Ann_{R_{1}}(N_{1})\times R_{2}=Ann_{R}(N)$ is a graded $2$-absorbing primary ideal of $R$.
\item This result holds similar to part (3).
\end{enumerate}
\end{proof}


\begin{thebibliography}{9}

\bibitem{Dawwas Bataineh} R. Abu-Dawwas, M. Bataineh and H. Shashan, Graded generalized $2$-absorbing submodules, Beitr\"{a}ge zur Algebra und Geometrie / Contributions to Algebra and Geometry, (2020), https://doi.org/10.1007/s13366-020-00544-1.

\bibitem{Zoubi Dawwas Ceken} K. Al-Zoubi, R. Abu-Dawwas and S. \c{C}eken, On graded 2-absorbing and graded weakly 2-absorbing ideals, Hacettepe Journal of Mathematics and Statistics, 48 (3) (2019), 724-731.

\bibitem{Zoubi Sharafat} K. Al-Zoubi and N. Sharafat, On graded $2$-absorbing primary and graded weakly $2$-absorbing primary ideals, Journal of the Korean Mathematical Society, 54 (2) (2017), 675-684.

\bibitem{Toroghy} H. Ansari-Toroghy and F. Farshadifar, Graded comultiplication modules, Chiang Mai Journal of Science, 38 (3) (2011), 311-320.

\bibitem{Ansari Farshadifar} H. Ansari-Toroghy and F. Farshadifar, On graded second modules, Tamkang Journal of Mathematics, 43 (4) (2012), 499-505.

\bibitem{Atani Saraei} S. E. Atani and F. E. K. Saraei, Graded modules which satisfy the gr-radical formula, Thai Journal of Mathematics, 8 (1) (2010), 161-170.

\bibitem{Ceken} S. \c{C}eken, $2$-Absorbing coprimary submodules, Beitr\"{a}ge zur Algebra und Geometrie / Contributions to Algebra and Geometry, (2021), https://doi.org/10.1007/s13366-020-00554-z.

\bibitem{Farzalipour} F. Farzalipour, P. Ghiasvand, On the union of graded prime submodules, Thai Journal of Mathematics, 9 (1) (2011), 49-55.

\bibitem{Hazart} R. Hazrat, Graded rings and graded Grothendieck groups, Cambridge University press, 2016.

\bibitem{Nastasescue}  C. Nastasescu and F. Oystaeyen, Methods of graded rings, Lecture Notes in Mathematics, 1836, Springer-Verlag, Berlin, 2004.

\bibitem{Refai Dawwas} M. Refai and R. Abu-Dawwas, On generalizations of graded second submodules, Proyecciones Journal of Mathematics, 39 (6) (2020), 1547-1564.

\bibitem{Refai Hailat} M. Refai, M. Hailat and S. Obiedat, Graded radicals and graded prime spectra, Far East Journal of Mathematical Sciences, (2000), 59-73.

\bibitem{Saber} H. Saber, T. Alraqad and R. Abu-Dawwas, On graded $s$-prime submodules, AIMS Mathematics, 6 (3) (2020), 2510-2524.

\end{thebibliography}
\end{document}